\documentclass[11pt]{article}
\usepackage{amsmath,amssymb,amsfonts,amsthm}
\usepackage{float}
\numberwithin{equation}{section}
\newtheorem{theorem}{Theorem}[section]
\newtheorem{definition}{Definition}[section]

\newtheorem{proposition}[theorem]{Proposition}
\newtheorem{corollary}[theorem]{Corollary}

\begin{document}
\begin{center}
{\Large{\textbf{Contemplating some invariants of the Jaco Graph, $J_n(1), n \in \Bbb N$}}} 
\end{center}
\vspace{0.5cm}
\large{\centerline{(Johan Kok, Susanth C)\footnote {\textbf {Affiliation of author(s):}\\
\noindent Johan Kok (Tshwane Metropolitan Police Department), City of Tshwane, Republic of South Africa\\
e-mail: kokkiek2@tshwane.gov.za\\ \\
\noindent Susanth C (Department of Mathematics, Vidya Academy of Science and Technology), Thalakkottukara, Thrissur-680501, Republic of India\\
e-mail: susanth\_c@yahoo.com}}
\vspace{0.5cm}
\begin{abstract}
\noindent Kok et.al. $[7]$ introduced Jaco Graphs (\emph{order 1}).  In this essay we present a recursive formula to determine the \emph{independence number} $\alpha(J_n(1)) = |\Bbb I|$ with, $\Bbb I =\{v_{i,j}| v_1 = v_{1,1} \in \Bbb I$ and $v_i = v_{i,j} = v_{(d^+(v_{m,(j-1)}) + m + 1)}\}.$ We also prove that for the Jaco Graph, $J_n(1), n \in \Bbb N$ with the prime Jaconian vertex $v_i$ the chromatic number, $\chi(J_n(1))$ is given by:
\begin{equation*} 
\chi(J_n(1))
\begin{cases}
= (n - i) + 1 , &\text {if and only if the edge $v_iv_n$ exists,}\\ \\ 
= n - i, &\text {otherwise.}
\end{cases}
\end{equation*} 
We further our exploration in respect of \emph{domination numbers, bondage numbers} and declare the concept of the \emph{murtage number}\footnote{In honour of U.S.R. Murty, co-author of [2].} of a simple connected graph $G$, denoted $m(G).$ We conclude by proving that for any Jaco Graph $J_n(1), n \in \Bbb N$ we have that $0 \leq m(J_n(1)) \leq 3.$
\end{abstract}
\noindent {\footnotesize \textbf{AMS Classification Numbers:} 05C07, 05C20, 05C38, 05C75, 05C85} 
\section{Introduction} 
Let $\mu(G)$ be an arbitrary invariant of the simple connected graph $G$. The \emph{$\mu$-stability number} of $G$ is conventionally, the minimum number of vertices whose removal changes $\mu(G).$ If the removal of the minimum vertices results in a decrease of the invariant the result is conventionally denoted, $\mu^-(G)$ and if the change is to the contrary the change is denoted $\mu^+(G).$ We note that the \emph{domination number}, $\gamma(G'),$ of a subgraph $G'$ of $G$ can be larger or smaller than $\gamma(G).$ Note that a subgraph may result from the removal of vertices and/or edges from $G$. Furthermore, we note that the removal of edges only  from the graph $G$ to obtain $G'$ can only result in $\gamma(G')\geq \gamma(G).$ The minimum number of edges whose removal from $G$ results in a graph $G'$ with $\gamma(G') > \gamma(G),$ is called the \emph{bondage number} $b(G)$, of $G$.
\section{Some invariants of a Jaco Graph, $J_n(1), n \in \Bbb N$}
The infinite directed Jaco graph (\emph{order 1}) was introduced in $[7],$ and defined by $V(J_\infty(1)) = \{v_i| i \in \Bbb N\}$, $E(J_\infty(1)) \subseteq \{(v_i, v_j)| i, j \in \Bbb N, i< j\}$ and $(v_i,v_ j) \in E(J_\infty(1))$ if and only if $2i - d^-(v_i) \geq j.$ The graph has four fundamental properties which are; $V(J_\infty(1)) = \{v_i|i \in \Bbb N\}$ and, if $v_j$ is the head of an edge (arc) then the tail is always a vertex $v_i, i<j$ and, if $v_k,$ for smallest $k \in \Bbb N$ is a tail vertex then all vertices $v_ \ell, k< \ell<j$ are tails of arcs to $v_j$ and finally, the degree of vertex $k$ is $d(v_k) = k.$ The family of finite directed graphs are those limited to $n \in \Bbb N$ vertices by lobbing off all vertices (and edges arcing to vertices) $v_t, t > n.$ Hence, trivially we have $d(v_i) \leq i$ for $i \in \Bbb N.$
\subsection{Independence number of a Jaco Graph, $J_n(1), n \in \Bbb N$}
Consider the underlying graph of the finite directed Jaco Graph, $J_n(1), n \in \Bbb N.$ Obviously the graph has vertices $v_1, v_2, v_3, ..., v_n.$ Because the independence number is defined to be the number of vertices in a maximum independent set $[1]$, it is optimal to choose non-adjacent vertices recursively, each of minimum indice. This observation leads to the next theorem. Observe that $v_{i,j} =v_i$ as calculated on the \emph{j-th step} of a recursive formula applied to the vertices of a simple connected graph.
\begin{theorem}
The cardinality of the set $\Bbb I =\{v_{i,j}| v_1 = v_{1,1} \in \Bbb I$ and $v_i = v_{i,j} = v_{(d^+(v_{m,(j-1)}) + m + 1)}\},$ derived from the underlying graph of the Jaco Graph $J_n(1), n \in \Bbb N$ is equal to the independence number, $\alpha(J_n(1)).$
\end{theorem}
\begin{proof}
Clearly for $J_1(1)$ the cardinality of $\Bbb I = \{v_1\}$ equals 1 and it is indeed the maximum independent set. It is equally easy to see that the set $\Bbb I = \{v_1\}$ is indeed a maximum independent set of $J_2(1)$ as well. Considering $J_3(1)$ the derived maximum independent set is, $\Bbb I =\{v_1, v_3\}.$ It easily follows that $v_3 = v_{3, 2} = v_{(d^+(v_1) +1 +1)} = v_{(d^+(v_{1,(2-1)}) +1 + 1)}.$ It follows that this maximum independent set (not unique) remains valid for $J_3(1), J_4(1), J_5(1).$ Hence, $\alpha(J_i(1) = 2,$ for $3 \leq i \leq 5.$\\ \\
Assume on the \emph{$\ell$-th step} we have the maximum independent set $\{v_1, v_3, v_6, ..., v_{(d^+(v_{m,(\ell -1)}) + m + 1)}\}$ in respect of the Jaco Graphs $J_i(1)$ for $k=(d^+(v_{m,(\ell -1)}) + m + 1) \leq i \leq k + d^+(v_k).$ \\ \\
Considering the Jaco Graph $J_{(k+d^+(v_k) + 1)}(1)$ will yield a maximum independent set, $\{ v_1, v_3, v_6, ...,\\ v_{(d^+(v_{m,(\ell -1)}) + m + 1)}, v_{(k + d^+(v_k) + 1)}\}.$ So the result holds for the \emph{($\ell$ + 1)-th step}. Through mathematical induction the result holds in general.
\end{proof}
\begin{corollary}
It follows that the \emph{covering number,} $\beta(J_n(1)) = n - \alpha(J_n(1)).$
\end{corollary}
\subsection{Chromatic number of a Jaco Graph, $J_n(1), n \in \Bbb J$}
From the definitions provided in $[7]$ the Hope Graph of the Jaco Graph, $J_n(1)$ is the complete graph on the vertices $v_{i+1}, v_{i+2}, ..., v_n$ if and only if $v_i$ is the prime Jaconian vertex of $J_n(1).$ Hence, $\Bbb H_n(1) \simeq K_{n- i}.$ The reader is reminded that a \emph{t-colouring} of a graph $G$ is a map $\lambda: V(G) \rightarrow [c]:= \{1,2,3, ..., c, c \geq0\}$ such that $\lambda(u) \neq \lambda(v)$  whenever $u, v \in V(G)$ are adjacent in $G$. The \emph{chromatic number} of $G$ denoted $\chi(G)$ is the minimum $c$ such that $G$ is \emph{c-colourable}. Now the following theorem can be settled.
\begin{theorem}
For the Jaco Graph, $J_n(1), n \in \Bbb N$ with the prime Jaconian vertex $v_i$ we have that the chromatic number, $\chi(J_n(1))$ is given by:
\begin{equation*} 
\chi(J_n(1))
\begin{cases}
= (n - i) + 1 , &\text {if and only if the edge $v_iv_n$ exists,}\\ \\ 
= n - i, &\text {otherwise.}
\end{cases}
\end{equation*} 
\end{theorem}
\begin{proof}
(a(i)) If the edge $v_iv_n$ exists the largest complete subgraph of $J_n(1)$ is given by $\Bbb H_n(1) + v_i \simeq K_{(n-i) +1}.$ Since it is known that $\chi(K_{(n-i) + 1}) = (n-i) + 1,$ it follows that $\chi(J_n(1)) \geq (n-i) + 1.$
For $J_1(1)$ we have that the prime Jaconian vertex is $v_1$ and inherently connected to itself. One may imagine the imaginary edge $"v_1v_1"$ to find $\chi(J_1(1)) = (1-1) + 1 = 1$ to be true. For $J_2(1)$ the prime Jaconian vertex is $v_1$ and the Hope Graph, $\Bbb H_2(1) \simeq K_1.$ Also, the edge $v_1v_2$, exists. Thus, $\chi(J_2(1)) = (2-1) + 1 = 2,$ which is true.\\ \\
Now assume the result holds for any $J_n(1), n > 2$ for which the edge $v_iv_n$ exists and $v_i$ is the prime Jaconian vertex. Label the $(n-i) + 1$ colours used to colour the vertices $v_i, v_{i+1}, v_{i+2}, ,,, v_n,$ consecutively, $c_i, c_{i+1}, c_{i+2}, ..., c_n.$ From definitions 1.3 and 1.4 and Lemma 1.1 $[7]$ it follows that if the prime Jaconian vertex $v_i$ is unique, the Jaco Graph $J_{n+1}(1)$ will be the \emph{smallest} Jaco Graph \emph{larger} than $J_n(1)$ with prime Jaconian vertex $v_{i+1}$ for which the edge $v_{i+1}v_{n+1}$, exists. It also implies that $\Bbb H_{n+1}(1) \simeq \Bbb H_n(1).$ Since the edge $v_iv_{n+1}$ does not exists, the colouring of $v_{n+1}$ with $c_1$ suffices, whilst the colouring of the rest of the graph $J_{n+1}(1)$ remains the same as that of $J_n(1).$ So clearly the result $\chi(J_{n+1}(1)) = ((n+1)-(i+1)) + 1 = (n-i) +1 = \chi(J_n(1))$ holds.\\ \\
From definitions 1.3 and 1.4 and Lemma 1.1 $[7]$ it follows that if the prime Jaconian vertex $v_i$ of $J_n(1)$ is not unique, the Jaco Graph $J_{n+2}(1)$ will be the \emph{smallest} Jaco Graph \emph{larger} than $J_n(1)$ with prime Jaconian vertex $v_{i+1}$ for which both the edge $v_{i+1}v_{n+1}$ and $v_{i+1}v_{n+2}$, exist (also see the Fisher Table for illustration). Since the edge $v_iv_{n+1}$ does not exist, colour vertices $v_{n+1}, v_{n+2}$ respectively $c_1$ and $c_{n+1}.$ Since $\Bbb H_{n+2}(1)$ has $(n - i) + 1$ vertices we must consider the colouring of $K_{(n-i) + 2}.$ We however, have that $\chi(K_{(n-i) + 2}) = (n-i) + 2 = ((n-i) + 1) + 1 = ((n+1) - i) + 1 = ((n+2) - (i+1)) + 1 = \chi(J_{n+2}(1)).$\\ \\
Assume that for some Jaco Graph $J_n(1)$ with the edge $v_iv_n$ existing we have that $\chi(J_n(1)) > (n-i) +1.$ Clearly this contradicts the definition on minimality of the colouring set so we safely conclude that $\chi(J_n(1)) \ngtr (n-i) +1.$\\ \\
\noindent Since all cases have been considered the necessary condition follows through mathematical induction.\\ \\
(a(ii)) Consider the converse statement namely, if $\chi(J_n(1)) = (n-i) + 1$ then the edge $v_iv_n$ exists and assume it is not true for some Jaco Graph $J_n(1)$ by assuming that the edge $v_iv_n$ does not exists. The Hope Graph $\Bbb H_n(1) \simeq K_{n-i}$ requires $n-i$ colours. Since, the edge $v_iv_n$ does not exists, colouring $v_i$ the same as $v_n$ will suffice. It implies that using $(n-i) + 1$ colours contradicts the definition on minimality of the colouring set. Hence, the sufficient condition follows thus, the result.\\ \\
(b)\footnote {Reader can formalise the proof as an exercise.} The result follows directly from the proof of result (a) and the definition on minimality of the colouring set.
\end{proof}
\subsection{Introduction to the murtage number $m(G)$ of a simple connected graph $G$}
Note that if vertices $u$ and $v$ are not adjacent in $G$, then $\gamma(G + uv) \leq \gamma(G).$ The significance of this concept becomes apparent in the application of domination theory. In a situation where a $\gamma$-set of a graph is to represent costly facilities in a network $N$, it may be preferable to establish additional links (edges) between vertices of $N$ rather than constructing facilities at all vertices of a $\gamma$-set.\\ \\
In order to calculate the murtage number of a graph we introduce the concept of a $d_{om}$-\emph{sequence} of a $\gamma$-set, $X_i$ of a graph. Label the vertices of $X_i$ such that $V(G)$ can be partitioned into sets $D_{1,i}, D_{2,i}, ..., D_{\gamma(G),i}$ such that $D_{j,i}$ contains the vertex $v_j \in X_i$ and vertices in $V(G) - X_i$ which are adjacent to $v_j$ and such that, $|D_{1,i}| \leq |D_{2,i}| \leq ..., \leq |D_{\gamma (G),i}|$ and $|D_{1,i}|$ is a minimum. We define a $d_{om}$-\emph{sequence} of the $\gamma$-set $X_i$ as $(|D_{1,i}|, |D_{2,i}|, ..., |D_{\gamma(G),i}|).$ Clearly a $\gamma$-set can have more than one $d_{om}$-\emph{sequence}. Assume $G$ has $k$ $\gamma$-sets namely $ X_1, X_2, ..., X_k$. Let $\theta = \emph{absolute}(min|D_{1,j}|)$ for some $X_j$. All $\gamma$-sets, $X_\ell$ for which firstly, $|D_{1,\ell}| = \theta$ (\emph{primary condition}) and secondly, $d(v_1, v_i)$ is minimum for all $v_i \in X_{\ell}$ (\emph{secondary condition}) is said to be \emph{compact} $\gamma$-sets. The partitioning described above in respect of a compact $\gamma$-set is called a \emph{murtage partition} of $V(G).$\\ \\
As example let us consider the path $P_4$ with vertices labelled from \emph{left to right} $v_1, v_2, v_3$ and $v_4$. Clearly the $\gamma$-set $\{v_2, v_3\}$ is a $\gamma$-set with the $d_{om}$-sequence = (2,2) and $d(v_2, v_3) = 1.$ The aforesaid set is however not a \emph{compact} $\gamma$-set because the set $\{v_1, v_3\}$ has $d_{om}$-sequence = (1, 3) meaning \emph{absolute}$(min|D_{1, i}|) =1 < 2$ which is primary in the definition. The fact that $d(v_1, v_3) = 2 > 1 = d(v_2, v_3)$ is secondary in the definition. The corresponding murtage partition of $V(P_4))$ is $\{\{v_1\}, \{v_2, v_3, v_4\}\}.$\\ \\
Another example will be considering the path $P_5$ with the vertices labelled \emph{left to right} $v_1, v_2, v_3, v_4$ and $v_5$. Clearly the sets $\{v_1, v_4\}, \{v_2, v_4\}$ are $\gamma$-sets. Both have $d_{om}$-sequence $(2, 3)$ with set $\{v_2, v_4\}$ providing $d(v_2, v_4) = 2$ hence compact, whilst the set $\{v_1, v_4\}$ provides $d(v_1, v_4) = 3$ hence, non-compact. The murtage partion associated with the compact $\gamma$-set $\{v_2, v_4\}$ is $\{\{v_1, v_2\}, \{v_3, v_4, v_5\}\}.$
\begin{definition}
We define the murtage number, $m(G)$, of a simple connected graph $G$ to be the minimum number of edges that has to be added to $G$ such that the resulting graph $G'$ has $\gamma(G') < \gamma(G).$
\end{definition}
\noindent It follows from the definition that $m(G) = 0$ if and only if $\gamma(G) = 1.$
\begin{theorem}
Let $|D_{1,i}| = \theta$ for some compact $\gamma$-set $X_i$ of G, then:
\begin{equation*} 
m(G)
\begin{cases}
= \theta , &\text {if and only if $v_1$ is not adjacent to any $v_j \in X_i,$}\\ \\ 
= \theta-1, &\text {if and only if $v_1$ is adjacent to some $v_j \in X_i$.}
\end{cases}
\end{equation*} 
\end{theorem}
\begin{proof}
(a) Assume $v_1$ is not adjacent to any $v_j \in X_i.$ Since we are considering a $d_{om}$-sequence of a compact $\gamma$-set of $G$, it is clear that the vertices in $D_{1,i}$ are uniquely dominated by $v_1$ hence, we must join all vertices in $D_{1,i}$ to vertices in $X_i - \{v_1\}$ in order to eliminate $v_1$ from $X_i$. Since, $|D_{1,i}| = \theta$ is an absolute minimum over all minimum number of edges to be added to have a resulting graph $G'$ such that $\gamma(G') =\gamma(G) - 1 < \gamma(G),$ it follows from the definition that $m(G) = \theta.$\\ \\
Conversely we assume that $m(G) = \theta$ and that $v_1$ is adjacent to some $v_j \in X_i.$ Since we are considering a $d_{om}$-sequence of a compact $\gamma$-set of $G$, it is clear that the vertices in $D_{1,i}$ are uniquely dominated by $v_1$ hence, we must join all vertices in $D_{1,i} - \{v_1\}$ to vertices in $X_i - \{v_1\}$ in order to eliminate $v_1$ from $X_i$. However, it required only $\theta - 1$ edges to be added hence, $m(G) = \theta -1.$ The latter is a contradiction, implying $v_1$ is not adjacent to any vertex $v_j \in X_i.$\\ \\
(b) The proof follows in a similar way as part (a).
\end{proof}
\begin{proposition}
For any graph $G$ for which $m(G) \geq 1$ we have that $m(G) = \gamma^-(G).$
\end{proposition}
\begin{proof}
Since $m(G) \geq 1$ it follows that $\gamma(G) \geq 2.$ Consider any compact $\gamma$-set $X_i$ of $G$. From the definition it follows that $m(G) = |D_{1,i}| = \theta.$ If $\gamma^-(G) = k < \theta$, let $Y \subseteq V(G)$ be a $\gamma^-$-set of $G$ with $|Y| = k.$ Since $\gamma(G - Y) < \gamma(G)$ there exists at least one vertex $v_j \in X_i$ such that every vertex of $G - (Y\cup X_i) \cup \{v_j\}$ is joined to a vertex in $X_i - \{v_j\}.$ Join every vertex in $Y$ to a vertex $v_t \in X_i, v_t \neq v_j$ to obtain $G'$. Clearly $\gamma(G') < \gamma(G)$ and it follows that $m(G) \leq k < \theta$, which is a contradiction.\\ \\
If $\theta < |Y| = \gamma^-(G)$ we consider the graph $G - D_{1,i}$ which has $\gamma$-set, $X_i - \{v_1\}.$ Since $\gamma(G - D_{1,i}) < \gamma(G)$ we have that $\gamma^-(G) \leq \theta < |Y|$ which renders a contradiction.\\ \\
Hence $m(G) = \gamma^-(G).$
\end{proof}
\noindent Although the two invariants differ conceptually, the result is very useful. We only have to investigate one of the invariants and all the results will hold for the other.
\begin{theorem}
Any simple connected graph $G$ has a spanning subtree $T$ such that:\\ $\Delta(T) = \Delta (G), \gamma(T) = \gamma(G)$ and $m(T) = m(G).$
\end{theorem}
\begin{proof}
Consider a compact $\gamma$-set, $X_i = \{v_1, v_2, v_3, ..., v_{\gamma(G)}\}$ of $G$ and an associated \emph{murtage partitioning} of $V(G).$ Consider the forest $\cup\langle D_{j,i}\rangle_{\forall j}$ with $\langle D_{j,i}\rangle$ the star with edges $\{v_jv_k| v_k \in D_{j,i}\}.$\\ \\
If in $\langle D_{\gamma(G),i}\rangle$ we have $d(v_{\gamma(G)}) = \Delta(G)$, then join all $\langle D_{j,i}\rangle, j = 1, 2, ..., (\gamma(G)-1)$ to $\langle D_{\gamma(G),i}\rangle$ with one edge $uv$ if and only if $u \in D_{\gamma(G),i}, v \in D_{j,i}$ and $uv \in E(G).$ Label the tree $T^*$. If any of the stars $\langle D_{j,i}\rangle$ has not been joined to $\langle D_{\gamma(G),i}\rangle$ we join them to $T^*$ with one edge $uv$ if and only if $u \in V(T^*), v \in D_{j,i}$ and $uv \in E(G).$ Label this successor tree $T^*.$ Since $G$ is connected it is evident that recursively all stars will eventually be connected. Clearly $\Delta(T) = \Delta(G).$\\ \\
If in $\langle D_{\gamma(G),i}\rangle$ we have $d(v_{\gamma(G)}) < \Delta(G),$  join all $\langle D_{j,i}\rangle, j = 1, 2, ..., (\gamma(G) - 1)$ to $\langle D_{\gamma(G),i}\rangle$ with one edge $uv_{\gamma(G)}$ if and only if $u \in D_{j,i}$ and $uv_{\gamma(G)} \in E(G).$ Label the tree $T^*$. Note that $\Delta(T^*) = \Delta(G).$ All other stars $\langle D_{j,i}\rangle$ which have not been joined at this first iteration can recursively be joined as described above. Hence, in all cases a spanning subtree $T$ can be constructed with $\Delta(T) = \Delta(G).$\\ \\
To complete the proof we note that $\gamma(G) \leq \gamma(T)$ and the set $X_i$ is a $\gamma$-set of $T$, hence $\gamma(T) = \gamma(G).$ It is also clear that $X_i$ is a compact $\gamma$-set of $T$ hence, $m(T) = m(G).$
\end{proof}
\noindent Furthermore, let $\Bbb G = \{G_1, G_2, G_3, ..., G_\ell\}$ with each $G_i,$ a simple connected graph. It follows easily that $\gamma(\cup_{\forall i}G_i) =\sum_{\forall i}\gamma(G_i)$ and similarly, $m(\cup_{\forall i}G_i) =\sum_{\forall i}m(G_i).$ Also if $\gamma(G_i) \leq \gamma(H_i), i = 1,2,3, ..., n$ then $\gamma(\cup_{\forall i}G_i) = \sum_{\forall i}\gamma(G_i) \leq \sum_{\forall i}\gamma(H_i) = \gamma(\cup_{\forall i}H_i).$
\subsection{Murtage number of a Jaco Graph, $J_n(1), n \in \Bbb N$}
In this subsection, reference to a Jaco Graph will mean we consider the undirected underlying graph of the Jaco Graph. Hence we \emph{peel off} the orientation of the Jaco Graph. From the definition of a Jaco Graph it follows that all Jaco Graphs on $n \geq 2$ has at least one leaf (vertex with degree = 1).  Hence, the \emph{bondage number} is $b(J_n(1))_{n \geq 2} = 1.$ \\ \\
The fact that $m(J_n(1))_{n\in \Bbb N} \geq 0$ follows from the definition.\\ \\
From the definition of a Jaco Graph it follows easily that vertex $v_1$ dominates $J_1(1)$ and $J_2(1)$ and vertex $v_2$ dominates $J_3(1)$ hence, $m(J_1(1)) = m(J_2(1)) = m(J_3(1)) = 0.$\\ \\
For $J_4(1)$ and $J_5(1)$  it follows that the set $\{v_1, v_3\}$ is a compact $\gamma$-set with the $d_{om}$-\emph{sequences}, $(1, 2)$ and $(1, 3)$ hence, $m(J_4(1)) = m(J_5(1)) =1.$ \\ \\
For the Jaco Graphs $J_6(1)$ and $J_7(1)$ we have  sets $\{v_1, v_4\}, \{v_1, v_5\}, \{v_2, v_4\}, \{v_2, v_5\}, \{v_2, v_6\},\\ \{v_2, v_7\}$ being $\gamma$-sets with only $\{v_2, v_4\}$ and $\{v_2, v_5\}$ the compact $\gamma$-sets. The corresponding $d_{om}$-\emph{sequences} are $(2, 4)$ and $(2, 5)$ hence, $m(J_6(1)) = m(J_7(1)) = 2.$ For $J_8(1)$ we have that the sets $\{v_2, v_5\}, \{v_2, v_6\}, \{v_2, v_7\}$ are $\gamma$-sets with $\{v_2, v_5\}$ the unique compact $\gamma$-set. The unique corresponding $d_{om}$-\emph{sequence} is (2, 6) so, $m(J_8(1)) = 2.$\\ \\
In respect of $J_9(1)$ and $J_{10}(1)$ we make the interesting observation that exactly two $\gamma$-sets, both being compact $\gamma$-sets namely, $\{v_2, v_6\}$ and $\{v_2, v_7\}$, exist. The corresponding $d_{om}$-\emph{sequences} are (3, 6) and (3, 7) respectively, meaning, $m(J_9(1)) = m(J_{10}(1)) = 3$.\\ \\
In the case of $J_{11}(1)$ an unique compact $\gamma$-set = $\{v_2, v_7\}$ exists with the $d_{om}$-sequence (3, 8). So also here we have $m(J_{11}(1)) = 3.$\\ \\
For $J_{12}(1)$ and $J_{13}(1)$ we note that the sets $\{v_1, v_3, v_8\}, \{v_1, v_3, v_9\}$ and $\{v_1, v_3, v_{10}\}$ are the $\gamma$-sets with $\{v_1, v_3, v_8\}$ the unique compact $\gamma$-set. The corresponding $d_{om}$-sequences are (1, 3, 8) and (1, 3, 9). Hence, $m(J_{12}(1)) = m(J_{13}(1)) = 1.$ Further exploratory analysis leads to the next theorem.
\begin{theorem}
For any Jaco Graph $J_n(1), n \in \Bbb N$ we have $0 \leq m(J_n(1)) \leq 3.$ The bounds are obviously sharp as well.
\end{theorem}
\begin{proof}
Following from the definition of a finite Jaco Graph $J_n(1), n \in \Bbb N$, it follows easily that the \emph{murtage number} can always be found be linking the minimum number of minimum (\emph{smallest}) indiced vertices labelled $v_i, i \in \{1, 2, 3, ....., k\}_{k < n}$ to some $v_j \in$ compact $\gamma$-set of $J_n(1).$\\ \\
Assume $m(J_n(1)) \geq 4.$ It implies that at least the vertices $v_1, v_2, v_3, v_4$ have to be linked to some vertex $v_j \in \gamma$-set, in order to reduce the value of $m(J_n(1))$ with at least 1. It also implies that $v_1, v_2, v_3, v_4 \notin$ compact  $\gamma$-set else $m(J_n(1)) \leq 3.$ Furthermore, the lowest indiced vertex $v_\ell \in$ compact $\gamma$-set is $4 <\ell = 8.$ However, the lowest indiced vertex dominated by $v_8$ is $v_5$ implying that vertices $v_1, v_2, v_3, v_4$ were not dominated, hence not adjacent to any vertex in the compact $\gamma$-set under consideration. The latter is a contradiction in terms of the definition of a $\gamma$-set (therefore, compact $\gamma$-set). So the result follows.
\end{proof}
\begin{corollary}
For any finite Jaco Graph $J_n(1), n \in \Bbb N$ we have that:\\
$ \gamma(J_n(1)) = \gamma(J_{(n - d^-(v_n) - d^-(v_{(n - d^-(v_n)}) -1)}(1)) + 1.$
\end{corollary}
\begin{proof}
Consider the Jaco Graph $J_n(1)$ and let vertex $v_\ell$ be the minimum indiced vertex with the edge $v_\ell v_n \in E(J_n(1)).$ Clearly all vertices $v_{k\neq \ell} \in \{v_{\ell - d^-(v_\ell)}, ..., v_n\}$ are adjacent to $ v_\ell$. Reducing by one more vertex we consider the Jaco Graph $J_{(n - d^-(v_n) - d^-(v_{(n - d^-(v_n)}) -1)}(1).$ Hence if $X_i$ is a compact $\gamma$-set of  $J_{(n - d^-(v_n) - d^-(v_{(n - d^-(v_n)}) -1)}(1),$ a compact $\gamma$-set of $J_n(1)$ is given by $X_i \cup \{v_\ell\}.$ \\ \\ 
It concludes the result that $ \gamma(J_n(1)) = \gamma(J_{(n - d^-(v_n) - d^-(v_{(n - d^-(v_n)}) -1)}(1)) + 1.$
\end{proof}
\noindent \textbf{\emph{Open access:\footnote {To be submitted to the \emph{Pioneer Journal of Mathematics and Mathematical Sciences.}}}} This paper is distributed under the terms of the Creative Commons Attribution License which permits any use, distribution and reproduction in any medium, provided the original author(s) and the source are credited. \\ \\
References (Limited) \\ \\
$[1]$ Bauer, D., Harary, F., Nieminen, J., Suffel, C., \emph{Domination alteration sets in graphs}, Discrete Mathematics, Vol 47 (1983), pp 153-161.\\
$[2]$ Bondy, J.A., Murty, U.S.R., \emph {Graph Theory with Applications,} Macmillan Press, London, (1976). \\
$[3]$ Dutton, R.D., Brigham, R.C., \emph{An extremal problem for edge domination insensitive graphs}, Discrete Applied Mathematics, Vol 20 (1988), no.2, pp 113-125.\\
$[4]$ Haynes, T.W., Henning, M.A., \emph{Changing and unchanging domination: a classification}, Discrete Mathematics, Vol 272 (2003), pp 65-79.\\
$[5]$ Haynes, T.W., Hedetniemi, S.M., Hedetniemi, S.T., \emph{Domination and independence subdivision numbers of graphs}, Discussioness Mathematicae Graph Theory, Vol 20 (2001), pp 271-280.\\ 
$[6]$ Kalayathankal, S.J., Susanth, C., \emph{The Sum and Product of Chromatic Numbers of Graphs and their Line Graphs}, arXiv: 1404.1698v1 [math.CO], 7 April 2014.\\
$[7]$ Kok, J., Fisher, P., Wilkens, B., Mabula, M., Mukungunugwa, V., \emph{Characteristics of Finite Jaco Graphs, $J_n(1), n \in \Bbb N$}, arXiv: 1404.0484v1 [math.CO], 2 April 2014.\\
$[8]$  Kok, J., Fisher, P., Wilkens, B., Mabula, M., Mukungunugwa, V., \emph{Characteristics of Jaco Graphs, $J_\infty(a), a \in \Bbb N$}, arXiv: 1404.1714v1 [math.CO], 7 April 2014. \\
$[9]$ Teschner, U., \emph{The bondage number of a graph}, Discrete Mathematics}, Vol 171 (1997), pp 249-259.
\end{document}